\documentclass[11pt,a4paper]{amsart}

\usepackage{amsmath,amsfonts,amsthm}

\newtheorem{lemma}{Lemma}

\title{Another Proof of a Lemma by L. Shepp}
\author{Tomas Persson}
\address{Tomas Persson, Centre for Mathematical Sciences, Lund University, Box 118, 22100 Lund, Sweden}
\email{tomasp@maths.lth.se}
\urladdr{http://www.maths.lth.se/~tomasp}


\begin{document}
\date{\today}
\begin{abstract}
  We give a new proof of a lemma by L. Shepp, that was used in
  connection to random coverings of a circle.
\end{abstract}

\maketitle

Consider a circle of circumference $1$, and a sequence $l_n$ of number
in $(0,1)$. We try to cover the circle by tossing arcs of length $l_n$
on the circle. In \cite{Shepp}, L. Shepp proved that if the arcs are
tossed independently and uniformly distributed on the circle, then the
circle is covered with probability one, if and only if $\sum n^{-2}
e^{l_1 + \cdots + l_n}$ diverges.  In the proof of this result, Shepp
used the following lemma.

\begin{lemma} \label{lem:shepp}
Let $l_n$ be a decreasing sequence of numbers in $(0,1)$, and\/ $0 <
\varepsilon < 1 - l_1$. Suppose that $\sum l_n^2$ diverges. Then
\[
\limsup_{n\to\infty} \int_0^\varepsilon \prod_{k=1}^n \frac{1 - l_k -
  \min\{l_k, t\}}{(1 - l_k)^2} \, \mathrm{d} t = \infty.
\]
\end{lemma}

Shepp's proof of this lemma is based on some related considerations of
probabilities, and he wrote that ``It seems difficult to prove
directly that $\sum l_n^2 = \infty$ implies that $\ldots$ holds.''
This induced T. Kaijser to look for a simple and direct proof, which
indeed he found in \cite{Kaijser}. In this note we provide yet another
proof. In fact, we shall prove the somewhat stronger statement that
\begin{equation} \label{eq:diverges}
  \lim_{n\to\infty} \int_0^\varepsilon \prod_{k=1}^n \frac{1 - l_k -
    \min\{l_k, t\}}{(1 - l_k)^2} \, \mathrm{d} t = \infty.
\end{equation}

We will make use of the following sometimes very useful inequality.
It might be well-known to the reader, but we provide nevertheless a
proof.

\begin{lemma} \label{lem:inequality}
  Let $f_1, \ldots, f_n$ be positive functions, that are either all
  increasing or all decreasing. Then
  \[
  \varepsilon^{n-1} \int_0^\varepsilon \prod_{k=1}^n f_k (x) \,
  \mathrm{d} x \geq \prod_{k=1}^n \int_0^\varepsilon f_k (x) \,
  \mathrm{d} x.
  \]
\end{lemma}

\begin{proof}
We have that
\[
f_1 (x) \geq f_1 (y) \quad \Longleftrightarrow \quad f_2 (x) \geq f_2 (y),
\]
since $f_1$ and $f_2$ are either both increasing or both
decreasing. Hence
\[
(f_1 (x) - f_1 (y))(f_2 (x) - f_2 (y)) \geq 0
\]
for all $x$ and $y$, and therefore
\[
\int_0^\varepsilon \int_0^\varepsilon (f_1 (x) - f_1 (y))(f_2 (x) -
f_2 (y)) \, \mathrm{d}x \, \mathrm{d}y \geq 0.
\]
This yields that
\[
\varepsilon \int_0^\varepsilon f_1 (x) f_2 (x) \, \mathrm{d}x \geq
\int_0^\varepsilon f_1 (x) \, \mathrm{d}x \int_0^\varepsilon f_2 (x)
\, \mathrm{d} x.
\]

So far, we have not used that the functions are positive, but this
will be used in the following step. Any product of the functions $f_1,
\ldots, f_n$ is monotone, and the proof is now finished by induction.
\end{proof}

We are now ready to prove \eqref{eq:diverges}. Put
\[
f_k (t) = \frac{1 - l_k - \min\{l_k, t\}}{(1 - l_k)^2}, \qquad k = 1, 2, \ldots
\]
When $l_k < \varepsilon$, a direct calculation shows that
\begin{equation} \label{eq:integrals}
  \int_0^\varepsilon f_k (t) \, \mathrm{d}t = \frac{ \frac{1}{2} l_k^2 +
  \varepsilon - 2 \varepsilon l_k}{(1 - l_k)^2}, \qquad k = 1, 2, \ldots
\end{equation}

We consider the function
\[
g_\varepsilon (x) = \frac{1}{\varepsilon} \frac{ \frac{1}{2} x^2 +
  \varepsilon - 2 \varepsilon x}{(1 - x)^2}.
\]
One easily checks that $g_\varepsilon (0) = 1$, $g_\varepsilon' (0) =
0$, and $g_\varepsilon'' (0) = \frac{1 - 2 \varepsilon}{\varepsilon}$.
Hence, we have that
\begin{equation} \label{eq:estimate}
g_\varepsilon (x) = 1 + \frac{1 - 2 \varepsilon}{2 \varepsilon} x^2 + o
(x^2).
\end{equation}

Since the functions $f_1, f_2, \ldots$ are all positive, it is
sufficient to prove \eqref{eq:diverges} for small $\varepsilon$. Hence
we may and will assume that $\varepsilon < \frac{1}{2}$ so that the
coefficient $\frac{1 - 2 \varepsilon}{2 \varepsilon}$ in
\eqref{eq:estimate} is positive.

Let $m$ be such that $l_k < \varepsilon$ for all $k > m$.  By
Lemma~\ref{lem:inequality} and \eqref{eq:integrals} we have for any $n
> m$ that
\[
\int_0^\varepsilon \prod_{k=1}^n f_k (t) \, \mathrm{d} t \geq
\varepsilon^{-m} \prod_{k=1}^{m-1} \int_0^\varepsilon f_k (t) \,
\mathrm{d} t \prod_{k=m}^n g_\varepsilon (l_k) = C \exp \biggl(
\sum_{k=m}^n \log g_\varepsilon (l_k) \biggr),
\]
where the positive constant $C$ does not depend on $n$. By
\eqref{eq:estimate} and $\sum_{k=1}^\infty l_k^2 = \infty$, we
conclude that \eqref{eq:diverges} holds.


\begin{thebibliography}{0}
\bibitem{Kaijser} T. Kaijser, {\em A Note on a Theorem by Larry
  Shepp}, Unpublished Report, Link\"oping University 1978,
  LiTH-MAT-R-78-18c.\\
  available at {\tt http://liu.diva-portal.org/smash/record.jsf?pid=diva2:764680}
\bibitem{Shepp} L. A. Shepp, {\em Covering the circle with random
  arcs}, Israel Journal of Mathematics 11 (1972), 328--345.
\end{thebibliography}
\end{document}